\documentclass[11pt]{amsart}
\usepackage{amssymb,amsfonts,amsmath}
\usepackage[all,2cell]{xy}
\usepackage{graphicx}
\usepackage[mathscr]{eucal}
\usepackage{enumerate}
\usepackage{indentfirst}
\usepackage{fancyhdr}

\theoremstyle{plain}
\newtheorem{thm}{\bf Theorem}[section]
\newtheorem{prop}[thm]{\bf Proposition}
\newtheorem{lem}[thm]{\bf Lemma}
\newtheorem{cor}[thm]{\bf Corollary}

\theoremstyle{definition}
\newtheorem{dfn}[thm]{\bf Definition}

\newtheorem{ex}[thm]{\bf Example}

\theoremstyle{remark}
\newtheorem{rem}[thm]{\bf Remark}

\def \A{\mathbb{A}}
\def \G{\mathbb{G}_m}

\def \k{\Bbbk}
\def \Z{\mathbb{Z}}
\def \Q{\mathbb{Q}}
\def \C{\mathbb{C}}
\def \R{\mathbb{R}}

\def \P{\mathbb{P}}

\newcommand{\acknowledge}{\subsection*{Acknowledgments}}
\newcommand{\convnotation}{\subsection*{Conventions and notation}}
\newcommand{\furnotation}{\subsection*{Further notation}}

\begin{document}

\title{   
Algebraic rational cells and   
equivariant intersection theory  
}
\author[Richard P. Gonzales]{Richard P. Gonzales\,*}
\thanks{* Supported by the Institut des Hautes \'{E}tudes Scientifiques,   
the Max-Planck-Institut f\"{u}r Mathematik, T\"{U}B\.{I}TAK 
Project No. 112T233, and DFG Research Grant PE2165/1-1.} 
\address{
Mathematisches Institut, 
Heinrich-Heine-Universit\"{a}t,   
40225 D\"{u}sseldorf, 
Germany}
\email{rgonzalesv@gmail.com}

\begin{abstract}
We provide 
a notion of algebraic rational cell 
with 
applications to
intersection theory 
on singular varieties with torus action. 
Based on this notion, 
we study $\Q$-filtrable varieties: 
algebraic varieties where a torus acts with 
isolated fixed points, such that   
the associated Bia{\l}ynicki-Birula decomposition 
consists of algebraic rational cells.
We show that the rational equivariant Chow group 
of any $\Q$-filtrable variety is  
freely generated by the classes of the cell closures. 
We apply this result to group embeddings,  
and more generally to spherical varieties. 
\keywords{Chow groups \and Torus actions \and Cell decompositions \and Algebraic monoids \and Spherical varieties}
\end{abstract}

\maketitle

\section{Introduction and motivation}
Let $\k$ be an algebraically closed field. 
The most commonly studied cell decompositions 
in algebraic geometry are those 
obtained by the method of Bia{\l}ynicki-Birula \cite{bb:torus}. 
If $\G\simeq \k^*$ acts on a smooth projective variety $X$  
with finitely many fixed points $x_1,\ldots, x_m$, then 
$X=\bigsqcup X_i$, where 
$$X_i=\{x\in X\;|\; \lim_{t\to 0}tx=x_i\}.$$ 
Moreover, the cells $X_i$ are isomorphic to affine spaces. 
From this one concludes e.g. that 
the Chow groups of $X$ are freely generated 
by the classes of the cell closures $\overline{X_i}\subseteq X$. 
This is quite notable, because   
the Chow groups of smooth varieties 
need not be finitely generated 
(consider e.g. a 
smooth projective curve of genus one).  
If $\k=\C$, then this decomposition   
implies that $X$ has no singular 
cohomology in odd degrees, and that 
the cycle map $cl_X:A_*(X)\to H_{*}(X)$ 
is an isomorphism, 
to mention 
just a few 
interesting applications. 
The {\em BB-decomposition} makes sense even 
if $X$ is singular, but the cells may no longer be so well-behaved.

In \cite{go:cells} we study the BB-decompositions of  
possibly singular complex projective varieties, 
assuming that 
the cells are rationally smooth (i.e. {\em rational cells}). 
Recall that a 
complex algebraic variety $X$, of dimension $n$,  
is called {\em rationally smooth}
if, for every $x\in X$, we have   
$$
\begin{array}{cccccc}
H^m(X,X-\{x\};\Q)=(0)\;\;{\rm if}\; m\neq 2n, 
&{\rm and} \;\; 
H^{2n}(X,X-\{x\};\Q)=\mathbb{Q}. 
\end{array}
$$ 
Such 
varieties satisfy Poincar\'e duality with rational coefficients.
If $X_i$ as above is a rational cell, 
then $\P(X_i):=(X_i\setminus\{x_i\})/\G$ 
is a rational cohomology complex projective space. 
Many important 
results on the equivariant cohomology of projective 
$T$-varieties 
admitting 
a BB-decomposition  
into rational cells 
are provided in 
\cite{go:cells}; for instance, such varieties 
have no cohomology in odd degrees and their 
equivariant cohomology is freely generated by the 
classes of the cell closures. 

\smallskip

The purpose of this paper is to provide analogues 
of such results
in the context of intersection theory for 
schemes with an action of a torus $T$ (i.e. $T$-schemes). 
For this, we introduce the notion of 
{\em algebraic rational cell}.  
Concisely, let $X$ be an affine $\G$-variety with 
an attractive fixed point $x$. Then $X$ is 
an algebraic rational cell if 
$\P(X):=[X\setminus \{0\}]/\G$
satisfies 
$$A_*(\P(X))_\Q\simeq A_*(\P^{n-1})_\Q,$$ 
where $n=\dim(X)$. 
The definition applies to actions of 
higher dimensional tori as well 
(Definition \ref{rationalcell.dfn}).  
Algebraic rational cells are modelled after 
(topological) rational cells \cite{go:cells}, 
although the resulting objects are not equivalent. 
In what follows, 
we show that 
algebraic rational cells 
are  
a good substitute for 
the notion of affine space in the 
study of Chow groups of singular varieties. 
This has applications to 
embedding theory (Section 5) 
and the geometry of spherical varieties (Section 6). 
In addition, some links  
between our present 
approach and that of \cite{go:cells} 
are built 
(Theorems \ref{main.monoid.thm},  \ref{ratsm.emb.are.qfilt.thm}, and \ref{sp_rs_impl_crs}). 
The techniques are mostly algebraic, and no essential use of the cycle map is made, except in Section 6. 

\smallskip
Here is an outline of the paper. 
Section 2 briefly 
reviews equivariant Chow groups of $T$-schemes. 
We also recall and discuss the notion of equivariant multiplicities 
at nondegenerate fixed points. 
The section concludes with some inequalities relating Chow groups and fixed point loci. 
In Section 3 we study the intersection-theoretical properties 
of algebraic rational cells (Proposition 
\ref{char.ratcell.prop}, Theorem \ref{char.thm}, Corollary \ref{gkm.cells.cor}). 
Next, in Section 4, we introduce the concept of (algebraically) 
$\Q$-filtrable spaces: projective $T$-varieties 
with isolated fixed points, such that 
the associated BB-decomposition is filtrable, and 
consists of algebraic rational cells (Definition \ref{qfiltrable.def}). 
The key result is 
Theorem 
\ref{qfilt.thm}. It asserts that the rational equivariant Chow group 
of any $\Q$-filtrable variety is freely generated by the classes of the cell closures. 

\smallskip

Having developed the 
theoretical framework for the study of 
$\Q$-filtrable varieties, 
we devote the last two sections to 
examples and applications. 
Let $G$ be a connected reductive group. 
Recall that a normal $G$-variety $X$ is called 
{\em spherical} if a Borel 
subgroup $B$ of $G$ has a dense orbit in 
$X$. Then it is known that $G$ and $B$ 
have finitely many orbits in $X$. 
It follows that $X$ contains only 
finitely many fixed 
points of a maximal torus $T\subset B$, see e.g.   
\cite{ti:sph}. These features make spherical varieties 
especially suitable for applying the techniques 
of this paper.

In Section 5 we apply our methods to a remarkable subclass of spherical varieties, namely,   
group embeddings. (We refer to  
that section for a definition of this key notion, and that of reductive monoids.)  
In this context, 
Theorem  \ref{main.monoid.thm}  
states that 
reductive monoids which are algebraic rational cells   
are characterized 
in the same way as rationally smooth monoids 
\cite{re:ratsm}.  
The second half of Section 5 deals with projective group embeddings (i.e. projectivizations of reductive monoids).  
The outcome (Theorem \ref{ratsm.emb.are.qfilt.thm}) 
provides an extension of   
\cite[Theorem 7.4]{go:cells}
to equivariant Chow groups. 


Finally, in Section 6, we study complex spherical varieties. The  
purpose there is to compare the two 
notions of $\Q$-filtrable varieties, the algebraic one (Section 4) 
and the topological one \cite{go:cells}. 
Roughly speaking, the main results  
of that section 
assert that {\em if} $X$ is a spherical $G$-variety 
which is $\Q$-filtrable in the sense of \cite{go:cells},  
then it is also $\Q$-filtrable in the sense of the present paper. 
Moreover, 
for such (possibly singular) $X$,    
the $T$-equivariant and non-equivariant cycle maps are isomorphisms. 
See Theorems \ref{sphrsm_implies_arc.lem} and \ref{sp_rs_impl_crs} for precise statements. 

\acknowledge{
The research in this paper was done 
during my visits to the 
Institute
des Hautes \'Etudes Scientifiques (IHES) and 
the 
Max-Planck-Institut f\"{u}r Mathematik (MPIM).    
I am deeply grateful to both institutions for their support,  
outstanding hospitality, and excellent working conditions. 
A very special thank you goes to Michel Brion 
for 
the productive 
meeting we had at IHES, from which this paper received much inspiration.  
I would also like to thank the support that I received, as a postdoctoral fellow, from 
Sabanc\i\,\"{U}niversitesi 
and 
the Scientific and Technological Research Council of Turkey (T\"{U}B\.ITAK), Project No. 112T233.
}

\section{Preliminaries} 

\convnotation{
Throughout this paper, we work over an algebraically closed 
field $\k$ of arbitrary characteristic (unless stated otherwise).  
All schemes and algebraic groups are assumed to be defined 
over $\k$. By a scheme we mean a separated scheme 
of finite type. 
A variety is a reduced scheme. 
Observe that varieties need not be irreducible. 
A subvariety is a closed subscheme which is a variety. 
A point on a scheme will always be a closed point. 

\smallskip

We denote by $T$ an algebraic torus. 
A scheme $X$ provided with an algebraic action of $T$ is called a 
{\em $T$-scheme}. 
For a $T$-scheme $X$,  
we denote by $X^T$ the fixed point subscheme and by $i_T:X^T\to X$
the natural inclusion. 
If $H$ is a closed subgroup of $T$, we similarly denote by 
$i_H:X^H\to X$ the inclusion of the fixed point subscheme. 
When comparing $X^T$ and $X^H$ we write $i_{T,H}:X^T\to X^H$ 
for the natural ($T$-equivariant) inclusion. 
For a scheme $X$, the dimension of the 
local ring of $X$ at $x$
is denoted  $\dim_x X$. 
We denote by $\Delta$ the character group of $T$, and by 
$S$ the symmetric algebra over $\Q$ of the abelian group 
$\Delta$. We denote by $\mathcal{Q}$ the quotient field of $S$.
Equivariant Chow groups 
are always considered with rational coefficients.

\smallskip

In this paper, torus actions are assumed to be 
{\em locally
linear}, i.e. the schemes we consider are covered by invariant affine open subsets.  
This assumption is fulfilled e.g. 
by $T$-stable subschemes of normal 
$T$-schemes \cite{su:eq}.
%
%
}

\subsection{The Bialynicki-Birula decomposition} 
The results in this subsection 
are due to Bialynicki-Birula \cite{bb:torus}, \cite{bb:decomp} (in the smooth case) 
and Konarski \cite{ko:bb} (in the general case). 
For our purposes, 
it suffices to consider the case of torus actions with isolated fixed points. 

Let $T$ be an algebraic torus. Let $X$ be a $T$-scheme 
with isolated fixed points.  
Then $X^T$ is finite and we
write $X^T=\{x_1,\ldots,x_m\}$. 
Recall that a one-parameter subgroup $\lambda:\G\to T$ is called {\em generic}
if $X^{\G}=X^T$, where $\G$ acts on $X$ via $\lambda$. 
Generic one-parameter subgroups always exist, 
due to local linearity of the action. 
Now fix a generic one-parameter subgroup $\lambda$ of $T$. 
For each $i$, define the subset 
$$X_+(x_i,\lambda)=\{x \in X \;|\; 
\lim_{t\to 0}\lambda(t)\cdot x=x_i\}.$$
Then $X_+(x_i,\lambda)$ is a locally closed 
$T$-invariant subscheme of $X$.  
The (disjoint) union of the 
$X_+(x_i,\lambda)$'s 
might not cover all of $X$, 
but when 
it does  
(e.g., when $X$ is complete), 
the decomposition $\{X_+(x_i,\lambda)\}_{i=1}^m$ 
is called 
%
%
the 
Bialynicki-Birula decomposition, or 
{\em BB-decomposition}, of $X$ associated to $\lambda$. 
Each $X_+(x_i,\lambda)$ is 
called 
a {\em cell} 
of the decomposition. 
%
Usually the $BB$-decomposition of a 
complete $T$-scheme is not a Whitney stratification; 
that is, it may happen that the closure of a cell  
is not a union of cells,  
even when the scheme is assumed to be smooth. 
For instance, see \cite[Example 1]{bb:decomp}.

\begin{dfn}\label{filtable.dfn} 
Let $X$ be a $T$-scheme with finitely many fixed points. 
Let  
$\{X_+(x_i,\lambda)\}_{i=1}^m$ be the BB-decomposition 
associated to some generic one-parameter subgroup $\lambda$ of $T$. 
The decomposition $\{X_+(x_i,\lambda)\}$ is said to be {\em filtrable} if 
there exists a finite increasing sequence 
$\Sigma_0\subset \Sigma_1\subset \ldots \subset \Sigma_m$
of $T$-invariant closed subschemes of $X$ such that:

\smallskip

\noindent a) $\Sigma_0=\emptyset$, $\Sigma_m=X$,

\smallskip

\noindent b) $\Sigma_{j}\setminus \Sigma_{j-1}$ 
is a cell of the decomposition $\{X_+(x_i,\lambda)\}$, 
for each $j=1,\ldots, m$. 
\end{dfn}  
We will refer to $\Sigma_j$ as the {\em $j$-th filtered piece} of $X$. 
In this context, it is common to say that $X$ is {\em filtrable}. 
If, moreover, 
the cells $X_+(x_i,\lambda)$ are 
isomorphic to affine spaces $\A^{n_i}$,  
then 
$X$ is called {\em $T$-cellular}. 
%
%

\begin{thm}[\cite{bb:torus}, \cite{bb:decomp}]\label{bbdecomp.thm}
Let $X$ be a complete $T$-scheme with isolated fixed points,  
and let $\lambda$ be a generic one-parameter subgroup. 
If $X$ admits an ample $T$-linearized invertible sheaf, 
then the associated BB-decomposition $\{X_+(x_i,\lambda)\}$ 
is filtrable.
Further, if $X$ is smooth,  
then $X$ is $T$-cellular. \hfill $\square$
\end{thm}


\subsection{Review of equivariant Chow groups. 
Localization theorem} 
Let $X$ be a $T$-scheme of dimension $n$ (not 
necessarily equidimensional). 
Let $V$ be a finite dimensional $T$-module, 
and let $U\subset V$ be an invariant open 
subset such that a principal bundle quotient 
$U\to U/T$ exists. 
Then $T$ acts freely on $X\times U$ 
and the quotient scheme 
$X_T:=(X\times U)/T$ exists. 
Following Edidin and Graham \cite{eg:eqint}, 
we define the 
$i$-th {\em equivariant Chow group} 
$A_i^T(X)$ by $A_i^T(X):=A_{i+\dim U-\dim T}(X),$
if $V\setminus U$ has codimension more than 
$n-i$. Such a 
pair $(V,U)$ always exist, and  
the definition is independent of the choice of $(V,U)$, see       
\cite{eg:eqint}. Finally, set $A^T_*(X)=\oplus_i A^T_*(X)$. 
If $X$ is a $T$-scheme, 
and $Y\subset X$ is a $T$-stable closed subscheme, 
then 
$Y$ defines a class $[Y]$ 
in $A^T_*(X)$. If $X$ is smooth, then 
so is $X_T$, 
and $A^T_*(X)$ admits an intersection pairing;  
in this case, denote by $A^*_T(X)$ 
the corresponding ring 
graded by codimension. The 
equivariant Chow ring $A^*_T(pt)$ identifies to $S$, 
and   
$A^T_*(X)$ is a $S$-module, where $\Delta$ acts on $A^T_*(X)$ 
by homogeneous maps of degree $-1$.   
This module structure 
is induced by pullback through the 
flat map $p_{X,T}:X_T\to U/G$.   
Restriction to a fiber of $p_{X,T}$ 
gives $i^*:A^T_*(X)\to A_*(X)$. 
If $X$ is complete, we denote by 
$\int_X(\alpha) \in S$ 
the proper pushforward to a point of a 
class $\alpha \in A^T_*(X)$. See \cite{eg:eqint} for details. 

\smallskip

Next we state 
Brion's description \cite{bri:eqchow} 
of the equivariant Chow groups 
in terms of  
invariant cycles. 
It also shows how to recover the 
usual Chow groups from equivariant ones.

\begin{thm}\label{Tequiv.thm}
Let $X$ be a $T$-scheme. Then the $S$-module $A^T_*(X)$ is defined by generators $[Y]$ where 
$Y$ is an invariant irreducible subvariety of $X$ and relations $[{\rm div}_Y(f)]-\chi[Y]$ where $f$ is a 
rational function on $Y$ which is an eigenvector of $T$ of weight $\chi$. Moreover, the 
map $A^T_*(X)\to A_*(X)$ vanishes on $\Delta A^T_*(X)$, and it induces 
an isomorphism 
$$
A^T_*(X)/\Delta A^T_*(X)\to A_*(X). 
$$\hfill $\square$
\end{thm}


The following is a slightly more 
general version of 
the localization theorem for equivariant Chow groups \cite[Corollary 2.3.2]{bri:eqchow}. 
For a proof, see e.g. \cite[Proposition 2.15]{go:tlin}. 

\begin{thm}\label{loc.chow.thm}
Let $X$ be a $T$-scheme, let $H\subset T$ be a closed subgroup, and 
let  
$i_H:X^H\to X$ be the inclusion of the fixed point subscheme.
Then the induced morphism of equivariant Chow groups 
$$i_{H*}:A^T_*(X^H)\to A^T_*(X)$$
becomes an isomorphism after inverting finitely many characters of $T$
that restrict non-trivially to $H$.  \hfill $\square$ 
\end{thm} 


Let $X$ be a $T$-scheme. In many situations, 
Theorems \ref{Tequiv.thm} and \ref{loc.chow.thm} combined 
yield a relation between the dimensions of the $\Q$-vector spaces 
$A_*(X)$ and $A_*(X^T)$. 

\begin{lem}\label{dim.fix.set.lem}
Let $X$ be a $T$-scheme. 
If $A_*(X)$ is a finite-dimensional $\Q$-vector space, 
then the inequality $\dim_\Q A_*(X^T)\leq \dim_\Q A_*(X)$ holds.  
Furthermore, $\dim_\Q A_*(X^T)=\dim_\Q A_*(X)$ if and only if 
the $S$-module $A^T_*(X)$ is free. 
\end{lem}

\begin{proof} 
The degrees in $A^T_*(X)$ 
are at most the dimension of $X$, 
so by the graded Nakayama lemma \cite[Exercise 4.6]{ei:comm}, 
the $S$-module $A^T_*(X)$ is finitely generated. 
The content of the corollary is now deduced from  
applying 
Lemma \ref{freeness.lem} and Remark \ref{ineq.rem} below  
to $M=A^T_*(X)$, taking into 
account that   
$\dim_{S/\mathfrak{m}}(M/\mathfrak{m}M)=\dim_\Q(A_*(X))$ 
(Theorem \ref{Tequiv.thm}), 
$\dim_{\mathcal{Q}}(M\otimes_S \mathcal{Q})=
\dim_{\mathcal{Q}}(A^T_*(X^T)\otimes \mathcal{Q})$ 
(Theorem \ref{loc.chow.thm}), and observing that 
this corresponds 
to $\dim_\Q A_*(X^T)$, since       
$A^T_*(X^T)=A_*(X^T)\otimes_\Q S$.  
\end{proof}


\begin{lem}\label{freeness.lem}
Let $S$ be a Noetherian positively graded ring 
such that $S_0$ is a field (e.g. $S=A^*_T(pt)$).  
Let $\mathfrak{m}$ be the unique graded maximal ideal and suppose 
$M$ is a non-zero finitely generated, graded, $S$-module. 
Suppose further that $S$ is an integral domain. Then   
$M$ is a free 
$S$-module  
if and only if
$$
\dim_{S/\mathfrak{m}}{(M/\mathfrak{m}M)}=\dim_{\mathcal{Q}}(M\otimes_{S}\mathcal{Q}), 
$$   
where $\mathcal{Q}$ is the quotient field of $S$. 
\end{lem}

\begin{proof}
If $M$ is free, then clearly the equation above holds. 
Conversely, denote by $n$ the common value of the two sides 
of the equation above. By the graded Nakayama lemma, $M$ 
has a system $\{x_1,\ldots, x_n\}$ of homogeneous generators. 
Now the elements $x_j\otimes 1$ generate the vector space 
$M\otimes_{S}\mathcal{Q}$ over $\mathcal{Q}$. But as by 
hypothesis this space is of dimension $n$ over $\mathcal{Q}$, the 
elements $x_j\otimes 1$ 
are linearly independent over $\mathcal{Q}$. 
It follows that the $x_j$ are linearly independent over $S$ and so 
they form a basis of $M$.  
\end{proof}

\begin{rem}\label{ineq.rem}
The proof of Lemma \ref{freeness.lem} shows that 
that if $M$ is a finitely generated, graded, $S$-module, 
then the inequality  
$$
\dim_{S/\mathfrak{m}}{(M/\mathfrak{m}M)}\geq \dim_{\mathcal{Q}}(M\otimes_{S}\mathcal{Q}) 
$$ 
holds, as 
we can refine the 
generating set $\{x_j\otimes 1\}$ to get a basis of 
$M\otimes \mathcal{Q}$. 
\end{rem}


An important class of schemes to which 
Lemma \ref{dim.fix.set.lem} applies   
is the class of $T$-linear schemes. 
These are the equivariant analogues  
of the linear schemes considered by 
\cite{jann:kth},
\cite{to:linear}, 
and   
\cite{jo:linear}. 
$T$-linear schemes have been studied in 
\cite{f:sph},  \cite{jk:chow}, \cite{ap:opk} 
and \cite{go:tlin}. 
Briefly, a $T$-linear scheme is a $T$-scheme 
that can be obtained by an 
inductive procedure starting with a finite 
dimensional $T$-representation, 
in such a way that the complement of a 
$T$-linear scheme equivariantly 
embedded in affine space is also a $T$-linear scheme, 
and any $T$-scheme which
can be stratified as a finite disjoint union of $T$-linear schemes  
is a $T$-linear scheme. See 
\cite{jk:chow} or \cite{go:tlin} 
for details. 
It is known 
that if $X$ is a $T$-linear scheme,  
then $A^T_*(X)_\Z$ is a finitely generated 
$S_\Z$-module, and $A_*(X)_\Z$ 
is a finitely generated abelian group 
(see e.g. \cite[Lemma 2.7]{go:tlin}).    
Below are the concrete examples we are interested in.   
For a proof of items (i)-(ii)  
see \cite[Proposition 3.6]{jk:chow}, 
for item (iii) see \cite[Theorem 2.5]{go:tlin}. 

\begin{thm}
Let $T$ be an algebraic torus. 
Then the following hold: 
\begin{enumerate}[(i)]
 \item A $T$-cellular scheme is $T$-linear.  
 \item Every $T$-scheme with finitely many $T$-orbits is $T$-linear. 
In particular, a toric variety with dense torus $T$ is $T$-linear. 
\item 
Let $B$ be a connected solvable linear algebraic group with maximal torus $T$.
Let $X$ be a $B$-scheme. If $B$ acts on $X$ with finitely 
many orbits, then $X$ is $T$-linear. In particular, 
spherical varieties are $T$-linear. \hfill $\square$
\end{enumerate} 
\end{thm}


\subsection{Nondegenerate fixed points and equivariant multiplicities}
Let $X$ be a $T$-scheme. 
A fixed point $x\in X$ 
is called {\em nondegenerate} 
if all weights of $T$ in the tangent space 
$T_xX$ are non-zero. 
A fixed point in a nonsingular $T$-variety is nondegenerate 
if and only if it is isolated. 
To study possibly singular schemes,  
Brion 
developed a notion of {\em equivariant multiplicity} at nondegenerate 
fixed points \cite{bri:eqchow}. 
The main features of this concept are outlined below. 

\begin{thm}[\protect{\cite[Theorem 4.2]{bri:eqchow}}]\label{eqmult.prop.thm}
Let $X$ be a $T$-scheme with an action of $T$, 
let $x\in X$ be a nondegenerate
fixed point and let $\chi_1,\ldots,\chi_n$ be the weights of $T_xX$ 
(counted with multiplicity).  
\begin{enumerate}[(i)]
 \item There exists a unique $S$-linear map 
$$e_{x,X}:A^T_*(X)\longrightarrow \frac{1}{\chi_1\cdots \chi_n}S$$
such that $e_{x,X}[x]=1$ and that $e_{x,X}[Y]=0$ for any 
$T$-invariant irreducible 
subvariety $Y\subset X$ which does not contain $x$.

\item For any $T$-invariant irreducible subvariety $Y\subset X$, the rational function 
$e_{x,X}[Y]$ is homogeneous of degree $-\dim(Y)$ and it coincides with 
$e_{x,Y}[Y]$. 

\item The point $x$ is nonsingular in $X$ if and only if $e_x[X]=\displaystyle \frac{1}{\chi_1\cdots\chi_n}.$ \hfill $\square$
\end{enumerate}
\end{thm}

For any $T$-invariant irreducible 
subvariety $Y\subset X$, we set $e_{x,X}[Y]:=e_x[Y]$, 
and we call $e_x[Y]$ the 
{\em
equivariant multiplicity of $Y$ at $x$}. 

\begin{prop}[\protect{\cite[Corollary 4.2]{bri:eqchow}}]\label{eqmult.loc.prop}
Let $X$ be a $T$-scheme such that all fixed 
points in $X$ are nondegenerate, and let 
$\alpha \in A^T_*(X)$. Then we have 
in $A^T_*(X)\otimes_S \mathcal{Q}$: 
$$\alpha=\sum_{x\in X^T}e_x(\alpha)[x].$$ 
\hfill $\square$
\end{prop}

Next we consider a  
special class of nondegenerate fixed points. 
Let $X$ be a $T$-variety. Call a fix point $x\in X$
{\em attractive} 
if all weights in the tangent space $T_xX$ are contained 
in some open half-space of $\Delta_\R=\Delta\otimes_{\Z} \R$, that is,  
some one-parameter subgroup of $T$ acts on $T_xX$ with positive weights only.  
Below is a characterization.  

\begin{thm}[\protect{\cite[Proposition A2]{bri:rat}}] \label{at.fx.p.thm}
Let $X$ be a $T$-variety with a fixed point $x$. The following 
conditions are equivalent: 
\begin{enumerate}[(i)]
 \item $x$ is attractive. 
 \item There exists a one-parameter subgroup $\lambda:\G\to T$ 
such that, for all $y$ in a neighborhood of $x$, 
we have $\displaystyle \lim_{t\to 0}\lambda(t)y=x$. 
\end{enumerate}
Moreover, if (i) or (ii) holds, then 
$x$ admits a unique open affine $T$-stable 
neighborhood in $X$, denoted $X_x$, 
and $X_x$ admits 
a closed equivariant embedding into $T_xX$. \hfill $\square$ 
\end{thm}

Let $X$ be a $T$-variety with an attractive fixed point $x$. 
Denote by $\chi_1,\ldots, \chi_n$ the weights of $T_xX$. 
Let $\Delta^*$ be the lattice of one-parameter subgroups of $T$, 
and let $\Delta^*_{\R}$ be the associated real vector space.  
Notice that the one-parameter subgroups $\lambda$ 
satisfying Theorem \ref{at.fx.p.thm} (ii) form the interior of a 
rational polyhedral cone $\sigma_x\subset \Delta^*_{\R}$,  
by setting 
$$\sigma_x:=\{\lambda \in \Delta^*_{\R}\;\;|\;\; \langle \lambda,\chi_i\rangle \geq 0 \;\,{\rm for}\;\, 1\leq i\leq n \}.$$ 
It follows from Theorem \ref{at.fx.p.thm} that $X_x$ equals 
$X_+(x,\lambda)$ 
for any $\lambda \in \sigma_x^0$. 

\begin{prop}[\protect{\cite[Proposition 4.4]{bri:eqchow}}]\label{attractive.prop}
Notation being as above, the rational function $e_x[X]$, 
viewed as a rational function on $\Delta^*_\R$, 
is defined at $\lambda$ and its 
value is the multiplicity of the algebra of 
regular functions on $X_x$ 
graded via the action of $\lambda$. In particular, 
$e_x[X]$ is non-zero. \hfill $\square$
\end{prop}

\subsection{Local study. Some inequalities relating Chow 
groups and fixed point loci}
Let $X$ be an affine $T$-variety with an attractive fixed point $x$. 
It follows from Proposition \ref{attractive.prop} that 
$X=X_+(x,\lambda)$
for any $\lambda\in \sigma_x^0$.  
Also, $\{x\}$ is the unique closed $T$-orbit in $X$, and  
$X$ admits a closed $T$-equivariant embedding into $T_xX$. 
Observe that $\dim_x X=\dim X$, because $x$ 
is contained in every 
irreducible component of $X$.  

Choose $\lambda\in \sigma_x^0$. 
Then all the weights of the $\G$-action 
on $T_xX$ via $\lambda$ are positive. Hence  
the geometric quotient 
$$
\P_{\lambda}(X):=(X\setminus\{x\})/\mathbb{G}_m
$$
exists and is a projective variety. In fact, 
it is a closed 
subvariety of the weighted projective space $\P_\lambda(T_xX)$. 
On the other hand, by \cite[Proposition A3]{bri:rat},  
there exists a $\G$-module $V$ 
and a finite equivariant surjective morphism 
$\pi:X\to V$ such that $\pi^{-1}(0)=\{x\}$ (as a set). 
This allows  
to estimate 
the size 
of the Chow groups of 
$\P_\lambda(X)$ in various cases.

\begin{lem} \label{min.size.chow.lem}
In the situation above, 
$\pi$ induces a 
surjection  
$$
\pi_*:A_k(\P_\lambda(X))\to A_k(\P_\lambda(V))
$$  
for all $k\geq 0$. Consequently, 
$A_k(\P_\lambda(X))\neq 0$ if  
$0\leq k\leq \dim{(X)}$, and 
$A_k(\P_\lambda(X))=0$ otherwise.  \hfill $\square$
\end{lem}

\begin{rem} \label{brion.arabia.map.rem} 
Clearly, $\pi_*:A_*(X)\to A_*(V)$ is also surjective. 
Observe that 
if $X$ is equidimensional  
and $d$ is the 
degree of $\pi$, then 
$$e_x[X]=d\cdot e_0[V],$$ 
where $e_x[X]$ (resp. $e_0[V]$) 
is the 
$\G$-equivariant multiplicity of $X$ at $x$ (resp. of $V$ at $0$) 
\cite[Proposition 4.3]{bri:eqchow}.  
\end{rem}

\smallskip

Now assume that  
$\dim_\Q{A_*(\P_\lambda(X))}<\infty$. 
We record below a few elementary inequalities:   
\smallskip

\begin{enumerate}[(1)]
\item $\dim{A_*(\P_\lambda(X))}\geq \dim{A_*(\P_\lambda(X)^T)}$, 
by Lemma \ref{dim.fix.set.lem}. 
\smallskip

\item Notice that $\P_\lambda(X)^T=\bigsqcup_{H} \P_\lambda(X^{H})$, 
where the union runs over 
all codimension-one subtori of $T$. 
In fact, by linearity of the action, 
there is only a finite 
collection of codimension-one subtori, 
say $H_1,\ldots, H_r$, for which $X^{H_j}\neq X^T$.
Thus    
$$\dim_\Q A_*(\P_\lambda(X)^T)=\sum_{H_j}\dim_\Q A_*(\P_\lambda(X^{H_j})).$$ 

\smallskip

\item 
Let $H_j$ be as in (2). 
We may also assume that $x$ is an attractive fixed point of $X^{H_j}$, 
for the action of $\G\simeq T/H_j$. 
Hence, as in Lemma \ref{min.size.chow.lem}, 
there is a $T$-equivariant finite surjective 
map $\pi_j:X^{H_j}\to V_j$, 
where $V_j$ is some $T$-module with a trivial action of $H_j$. 
Thus 
$\dim_\Q A_*(\P_\lambda(X^{H_j})) \geq \dim_\Q A_*(\P_\lambda(V_j))$, 
which in turn yields 
$$ \sum_{j=1}^r\dim_\Q A_*(\P_\lambda(X^{H_j})) \geq \sum_{j=1}^r\dim_\Q A_*(\P_\lambda(V_j)).$$
Equality holds if and only if the  $\pi_j$'s 
induce isomorphisms on the Chow groups.  

\smallskip

\item Since each $\P_\lambda(V_j)$ in (3) is a weighted projective space, 
we get 
$$
\sum_{j=1}^r\dim_\Q A_*(\P_\lambda(V_j))=\sum_{j=1}^r\dim V_j=
\sum_{j=1}^r\dim X^{H_j},$$ 
where the last equality stems from the fact that each 
$\pi_j$ is finite and surjective. 

\smallskip

\item \label{sync.item} Because $x$ is an attractive fixed point, 
we have $$\sum_{j=1}^r\dim X^{H_j}\geq \dim X,$$
by \cite[Theorem 1.4]{bri:rat}. The 
equality holds if and only if there is a $T$-module $V$  
and a 
$T$-equivariant finite surjective 
morphism $\pi:X\to V$ such that $\pi^{-1}(0)=\{x\}$. 
To show this, we follow closely the argument in \cite[proof of Theorem 1.2]{bri:rat}. 
First, assume that such a morphism $\pi$ exists. 
Let $H_1,\ldots, H_r$ be subtori as in (2). 
Consider $H_j$ and a point $y\in V^{H_j}$. 
Since $H_j$ is connected, 
it acts trivially on the (finite) fiber $\pi^{-1}(y)$. 
This implies that the induced $T$-equivariant map $\pi_j:X^{H_j}\to V^{H_j}$ is  
finite and surjective. Hence, 
$\dim X=\dim V=\sum \dim V^{H_j}=\sum \dim X^{H_j}$. 
Conversely, if the equality $\dim X=\sum_j \dim X^{H_j}$ holds, 
then we can synchronize the 
maps   
$\pi_j$ from (3) as follows. 
%
%
%
Given that $X^{H_j}$ is $T$-stable and closed in $X$, 
we can extend $\pi_j$ to an
equivariant morphism 
$\pi_j':X\to V_j$. 
Let $V$ denote the product of all the $V_j$, and let $\pi:X\to V$ 
be the product morphism. 
Then $p(x)=0$ and $V^T=\{0\}$, by construction.  
Notice that $x$, being an attractive fixed point, lies in the closure all the $T$-orbits in $X$. 
In particular, $x$ is contained in all the irreducible components of $\pi^{-1}(0)$ (i.e. $\pi^{-1}(0)$ 
is connected). We claim that the map $\pi$ is finite. 
Indeed, $\{x\}=\pi^{-1}(0)$, for otherwise, $\pi^{-1}(0)$ would contain a $T$-stable curve upon 
which $T$ acts through a non-trivial character \cite[Proposition A.4]{bri:rat}. 
But this is impossible, because $\pi$ restricts to a
finite morphism on each $X^{H_j}$.
Finally, recall that $\dim V =\sum_j V_j=\sum_j X_j=\dim X$, by construction.   
Thus  
the map $\pi$ is dominant, and hence surjective.  

\end{enumerate}
\smallskip

Combining items (1) to (5), we obtain the 
chain of inequalities 
{\small
$$
\dim_\Q A_*(\P_\lambda(X))\geq \sum_{j=1}^r\dim_\Q A_*(\P_\lambda(X^{H_j}))\geq 
\sum_{j=1}^r\dim_\Q A_*(\P_\lambda(V_j))=\sum_{j=1}^r\dim X^{H_j}\geq \dim X.  
$$
}


\begin{prop}\label{first.char.rcells.prop}
Let $X$ be an affine 
$T$-variety with an attractive fixed point $x$.  
Let $H_1,\ldots,H_r$ 
denote the finite list of all 
codimension one subtori satisfying $X^{H_j}\neq X^T$.  
\begin{enumerate}[(a)]
\item The following are equivalent. 
\smallskip 

\begin{enumerate}[(i)] 
\item $\dim X=\sum_{j=1}^r \dim X^{H_j}$. 
\smallskip 

\item There is a $T$-module $V$ and a $T$-equivariant 
finite surjective morphism $\pi:X\to V$ such that 
$\pi(x)=0$ and $V^T=\{0\}$. In particular, for all $1\leq j\leq r$, the restriction of $\pi$ to $X^{H_j}$, denoted $\pi_j$,   
induces a $T$-equivariant finite surjective morphism $\pi_j:X^{H_j}\to V_j$, where $V_j:=V^{H_j}$.  
\end{enumerate}
\smallskip

\item Let $\lambda\in \sigma_x^0$. If $\dim_\Q A_*(\P_\lambda(X))=\dim X$, 
then conditions (i) and (ii) of (a) hold. 
Moreover, there is a chain of equalities 
{\small
 $$
\dim_\Q A_*(\P_\lambda(X))=\sum_{j=1}^r\dim_\Q A_*(\P_\lambda(X^{H_j}))= 
\sum_{j=1}^r\dim_\Q A_*(\P_\lambda(V_j))=\sum_{j=1}^r\dim X^{H_j}=\dim X,   
$$
}
and the maps $\pi$ and $\pi_j$ from (a) 
induce isomorphisms 
$$\pi_*:A_k(\P_\lambda(X))\xrightarrow{\sim} A_k(\P_\lambda(V)),$$ 
$${\pi_j}_*:A_k(\P_\lambda(X^{H_j}))\xrightarrow{\sim} A_k(\P_\lambda(V_j)),$$ 
for all $j$ and $k$. 
\end{enumerate}
\end{prop}


Put in perspective, this result is our motivation 
for the material in the next section.  

\section{Algebraic rational cells}
This section is devoted to the study of our main 
technical tool: algebraic rational cells. 
We thank M. Brion for leading us to the following definition. 

\begin{dfn}\label{rationalcell.dfn}
Let $X$ be an affine 
$T$-variety with an attractive fixed point $x$, and let 
$n=\dim X$. 
We say that $(X,x)$, or simply $X$, 
is an {\em algebraic rational cell}
if and only if, for some $\lambda \in \sigma_x^0$, we have   
$$
A_k(\P_\lambda(X))=
\left\{ \begin{array}{ccl} \Q \;&         &{\rm if }\;\;\; 0 \leq k \leq n-1, \\
                            0 \;&         &{\rm otherwise.} \\
\end{array}
\right.
$$ 
We abbreviate this condition by writing 
$A_*(\P_\lambda(X))\simeq A_*(\P^{n-1}).$ 
\end{dfn}
Algebraic rational cells are 
such $T$-varieties for which 
Proposition \ref{first.char.rcells.prop} (b)
holds.   
In principle, 
Definition \ref{rationalcell.dfn} depends 
on a particular choice of $\lambda\in \sigma_x^0$. 
But, as we shall see next, 
it is independent of $\lambda$:     
if $A_*(\P_\lambda(X))\simeq A_*(\P^{n-1})$ holds for 
some $\lambda \in \sigma_x^0$, then it holds 
for all $\lambda\in \sigma_x^0$.


\begin{lem}\label{contractible.lem}
Let $X$ be an affine 
$T$-variety with an attractive fixed point $x$, and let 
$n=\dim X$. 
Then $(X,x)$ is an algebraic rational cell if and only if  
$$A_k(X)=\left\{ \begin{array}{ccl} \Q \;&\; {\rm if }&\; k=n \\
                                      0 \;&\; {\rm if }&\; k\neq n.\\
\end{array}
\right.
$$ 
In particular, if $(X,x)$ is an algebraic rational cell, then 
it is irreducible.  
\end{lem}

\begin{proof} 
Let $\G$ act on $X$ via $\lambda$. 
Recall that we have a short exact sequence 
$$0\to A^{\G}_*(x)\to A^{\G}_*(X)\to A^{\G}_*(X\setminus \{x\})\to 0,$$
which stems from the localization theorem (Theorem \ref{loc.chow.thm}). 
As in Subsection 2.4, 
there exists a $\mathbb{G}_m$-equivariant finite surjective 
map $\pi:X\to \mathbb{A}^n$ such that $\pi^{-1}(0)=x$, and $\G$-acts on $\A^n$ with positive weights only.  
This map induces the commutative diagram:
$$
\xymatrix{
0 \ar[r]& A^{\G}_*(x) \ar[r]^{i_*} \ar[d]^{\pi_*}& A^{\G}_*(X) \ar[r]^{j^*} \ar[d]^{\pi_*} &  A^{\G}_*(X\setminus\{x\}) \ar[d]^{\pi_*}  \ar[r]  & 0 \\
0 \ar[r]& A^{\G}_*(0) \ar[r]^{i^*} & A^{\G}_*(\A^n)\ar[r]^{j^*}& A^{\G}_*(\A^n\setminus\{0\}) \ar[r]  & 0.\\
}
$$
The left vertical map is clearly an isomorphism. 
We claim that the other two vertical maps are surjective. 
Indeed, since $\pi:X\to \A^n$ is finite and surjective,  
the induced map of mixed spaces  
$\pi:X_{\G} \to {\A}^n_{\G}$ inherits both properties, by descent \cite[Propositions 2 and 3]{eg:eqint}.
Hence, $\pi_*:A_*^{\G}(X)\to A^{\G}_*(\A^n)$ is surjective. 
For the right vertical map, observe that $A^{\G}_*(X\setminus\{x\})\simeq A_*(\P(X))$ by \cite[Theorem 3]{eg:eqint}.  
So this map represents 
$\pi_*:A_*(\P(X))\to A_*(\P^{n-1}),$ 
whose surjectivity is already known (Lemma \ref{min.size.chow.lem}). 
%
We conclude from 
the previous analysis 
that the right vertical map is an isomorphism 
if and only if so is the middle one.   
But 
the latter happens if and only if 
$A_*(X)\simeq A_*(\A^n)\simeq \Q$ (one direction is guaranteed by Theorem \ref{Tequiv.thm}; for the other one use Lemma \ref{dim.fix.set.lem}).
This yields the first assertion of the lemma. 

Finally, the second assertion follows from Lemma \ref{aux.lem} below. 
\end{proof}

\begin{lem}[\protect{\cite[Proposition 1.6]{eh:book}}\label{aux.lem}]
Let $X$ be a variety. Let $X_1,\ldots, X_m$ be 
the irreducible components of $X$. Then 
the classes $[X_i]\subset A_*(X)_\Z$ 
generate a free abelian subgroup of rank $m$ in $A_*(X)_\Z$. Here $A_*(X)_\Z$ denotes the integral Chow group of $X$. 
\end{lem}

\begin{proof}
If $X$ is equidimensional, then the result is well-known, see e.g. \cite[Section 1.5]{f:int}. 
If $X$ is not equidimensional, then we argue as follows. 
Let $n=\dim X$. For $0\leq j\leq n$, let $Y_j$ be the union 
of those $X_i$'s whose dimension is exactly $j$. 
By definition, $Y_j$ is equidimensional. If $Y_j\neq \emptyset$, then 
$A_j(Y_j)=\Z^{c_j}$, where $c_j$ is the number of $X_i$'s whose dimension is $j$. 
Moreover, if $j\neq k$, then $\dim Y_j\cap Y_k< \min\{j,k\}$. 
Now, for $0\leq j\leq n$, define $Z_j:=\bigcup_{i\neq j}Y_i$. Notice that $\dim Y_j\cap Z_j<j$. 
By \cite[Example 1.8.1]{f:int} we have an exact sequence 
$$A_j(Y_j\cap Z_j)\to A_j(Y_j)\oplus A_j(Z_j)\to A_j(X)\to 0.$$
But $A_j(Y_j\cap Z_j)=0$ because $\dim Y_j\cap Z_j<j$. 
Thus $A_j(Y_j)\oplus A_j(Z_j)\simeq A_j(X)$. 
So $\bigoplus_{j=0}^nA_j(Y_j)\subseteq A_*(X)$. But $\bigoplus_{j=0}^nA_j(Y_j)=\Z^{m}$, 
since 
$m=\sum_j c_j$. The proof is now complete. 
\end{proof}

Lemma \ref{contractible.lem} hints to a more general structural property
of algebraic rational cells, with respect to the $T$-action.

\begin{prop}\label{char.ratcell.prop}
Let $X$ be an affine $T$-variety 
with an attractive fixed point $x$, and let $\lambda\in \sigma_x^0$. 
Let $n=\dim{X}$. Then the following conditions 
are equivalent.
 
\begin{enumerate}[(i)]
\item $A_*(\P_\lambda(X))\simeq A_*(\P^{n-1})$.
\smallskip
\item $A_*(X)\simeq A_*(\A^n)$. 
\smallskip
\item $A^T_*(X)\simeq A^T_*(pt)=S$. 
\smallskip
\item $A^T_*(\P_\lambda(X))\simeq A_*(\P^{n-1})\otimes_\Q S$
\end{enumerate}
\end{prop}

\begin{proof}
The equivalence 
(i) $\Leftrightarrow$ (ii) follows from 
Lemma \ref{contractible.lem}. 

The equivalence (ii) $\Leftrightarrow$ (iii) follows from 
Lemma \ref{dim.fix.set.lem}.

The implication (iv) $\Rightarrow$ (i) 
is deduced from Theorem \ref{Tequiv.thm} and   
Lemma \ref{min.size.chow.lem}.   

Finally, we dispose of the direction (i) $\Rightarrow$ (iv). 
Recall that (i) yields the existence of a 
$T$-equivariant finite surjective 
morphism $\pi:X\to V$, such that the induced map   
$\pi_*:A_*(\P_\lambda(X))\to A_*(\P_\lambda(V))$ 
is an isomorphism (Proposition \ref{first.char.rcells.prop} (b)).  
By the graded Nakayama lemma,  
the corresponding map 
$\widetilde{\pi_*}:
A^T_*(\P_\lambda(X))\to A^T_*(\P_\lambda(V))$ is 
surjective. 
We claim that 
$\widetilde{\pi_*}$ is also injective 
(hence an isomorphism). Indeed, 
choose a basis $z_1,..,z_n$ of $A_*(\P_\lambda(V))$. Now 
identify that basis with a basis of $A_*(\P_\lambda(X))$, 
via $\pi_*$, 
and lift it to a generating system of 
the $S$-module $A^T_*(\P_\lambda(X))$. 
This generating system is a basis, since its image
under $\widetilde{\pi_*}$ is a basis of $A^T_*(\P_\lambda(V))$. 
\end{proof}


Next, we exhibit some additional features 
of algebraic rational cells. The result is 
an algebraic counterpart of \cite[Theorem 18]{bri:ech}. 

\begin{thm}\label{char.thm}
Let $X$ be an irreducible affine $T$-variety with 
an attractive fixed point $x$. Then the following are equivalent. 
\begin{enumerate}[(i)]
\item For each subtorus $H\subset T$ of codimension one,   
$(X^{H},x)$ is an algebraic rational cell,  
and $\dim{X}=\sum_{H}\dim {X^{H}}$ 
(sum over all subtori of codimension one). 

\item For each subtorus $H\subset T$ of codimension one, 
$(X^{H},x)$ is an algebraic rational cell, and  
$$e_x[X]=d\prod_{H}e_x[X^{H}],$$ 
where $d$ is a positive  
rational number. 
If moreover each $X^{H}$ is smooth, then $d$ is an integer.   
\end{enumerate}  
Furthermore, if $(X,x)$ is an algebraic rational cell, then conditions (i) and (ii) hold. 
\end{thm}

\begin{proof}
Recall that there is only 
a finite collection of codimension one 
subtori, say $H_1,\ldots, H_r$, for which 
$X^{H_j}\neq X^T$. 
The required equivalence is obtained  
arguing exactly as in \cite[Theorem 18]{bri:ech}. 
Indeed, if (i) holds, then there exists a $T$-equivariant 
finite surjective map 
$\pi:X\to V$, where $V$ is a $T$-module   
(by Proposition \ref{first.char.rcells.prop} (a)).  
So $e_x[X]=de_0[V]$, where $d=\deg{\pi}$. 
But then $de_0[V]=d\prod_{H_j}e_0[V_j]$, 
because $V$ is a $T$-module 
(Theorem \ref{eqmult.prop.thm} (iii)). 
In turn, the last expression identifies to 
$\frac{d}{\prod_j d_j}\prod_{j} e_x[X^{H_j}]$, 
where 
$d_j=\deg{\pi_j}$  
and $\pi_j$ is as in Proposition \ref{first.char.rcells.prop} (a).   
Condition (ii) is thus attained. Conversely, if (ii) holds, 
then the $X^{H_j}$'s are irreducible 
(Lemma \ref{contractible.lem}). As $X$ is irreducible by assumption,  
the equality 
$e_x[X]=d\prod_{H_j}e_x[X^{H_j}]$ 
yields  
$\dim X=\sum_{H_j} X^{H_j}$ by Theorem \ref{eqmult.prop.thm} (ii). 

Finally, if $(X,x)$ is an algebraic rational cell, 
then condition (i) is deduced at once from   
%
Proposition \ref{first.char.rcells.prop} (b). 
\end{proof}

In general, it is not true that 
properties (i) or (ii) of Theorem \ref{char.thm} 
characterize algebraic rational cells.  
Here is an example, cf. \cite[Remark 1.4]{bri:rat}. 

\begin{ex}
Let $X$ be the hypersurface of $\A^5$ with equation $x^2+yz+xtw=0$. 
Note that $X$ is irreducible, with singular locus 
$x=y=z=tw=0$, a union of two lines meeting at the origin. 
Now consider the $\G\times \G$-action on $\A^5$ given by 
$(u,v)\cdot (x,y,z,t,w):=(u^2v^2x,u^3vy,uv^3z,u^2t,v^2w)$. Then the origin of 
$\A^5$ is an attractive fixed point, $X$ is $T$-stable of dimension four and 
$X$ contains four closed irreducible $T$-stable curves, namely, the coordinate 
lines except for the $x$-axis. 
So $X$ satisfies condition (i) of Theorem \ref{char.thm}. 
Nevertheless, 
$(X,0)$ 
is not an algebraic rational cell. 
To see this, consider the $\G$-action on $\A^5$ given by 
$u\cdot (x,y,z,t,w):=(x,uy,u^{-1}z,t,w)$. Then $X$ is $\G$-stable 
and $X^{\G}$ is defined by $y=z=x^2+xtw=0$. Thus $X^{\G}$ is reducible 
at the origin. In fact $A_*(X^{\G})=\Q\oplus \Q$ (since $X^{\G}$ consists of the 
union of two copies of $\A^2$).  
Thus $\dim_\Q A_*(X^{\G})=2$, and so    
$\dim_\Q A_*(X)\geq 2$, by Lemma \ref{dim.fix.set.lem}. 
Therefore, in view of 
Lemma \ref{contractible.lem}, 
$(X,0)$ is not an algebraic rational cell.
\end{ex}



\begin{ex}[Smooth rational cells]
Let $X$ be an affine $T$-variety with an attractive fixed point $x$. If $X$ is smooth at $x$, then $X\simeq T_xX$, 
$T$-equivariantly. Thus $(X,x)\simeq (T_xX,0)$ 
is an algebraic rational cell. This agrees with 
the fact that $\P_\lambda(T_xX)$ is a weighted 
projective space. 
\end{ex}

\begin{rem}
Let $\k=\C$. In general, there seems to be no immediate 
relation between algebraic rational cells and 
(topological) rational cells \cite{go:cells}.  
Nevertheless, in some important cases, e.g. spherical varieties, 
strong connections exist, cf. Sections 5 and 6.  
\end{rem}

We conclude this section by computing equivariant 
mutiplicities of algebraic rational cells. 
Recall that a primitive character $\chi$ of $T$ is called {\em singular} 
if $X^{\ker{(\chi)}}\neq X^T$. 

\begin{cor}\label{inv.poly.cor} \label{gkm.cells.cor}
Let $X$ be an irreducible $T$-variety with attractive fixed point $x$. 
Let $X_x$ be the unique open affine $T$-stable neighborhood of $x$. 
If $(X_x,x)$ is an algebraic rational cell, then the 
following hold: 
\begin{enumerate}[(i)]
\item 
$e_x[X]$ is the inverse of a polynomial. In fact, 
$$e_x[X]=\frac{d}{\chi_1 \cdots \chi_{r}},$$ 
where the $\chi_i$'s are singular characters, $r=\dim X$,  
and $d$ is a positive rational number.      
%
%
%

\item Additionally, if  
the number of closed irreducible 
$T$-stable curves through $x$ is finite, 
say $\ell(x)$,  
then $\dim X=\ell(x)$. Furthermore, 
we may take for $\chi_1,\ldots, \chi_r$ 
the characters associated to these curves. 
\end{enumerate} 
\end{cor}

\begin{proof}
Replacing $X$ by $X_x$ we may assume that $X$ is affine.
%
Then (i) 
follows at once from Theorem \ref{char.thm} and its proof.  
As for (ii) simply use 
Theorem \ref{char.thm} and Corollary \ref{inv.poly.cor}, 
to adapt the argument of 
\cite[Corollary 1.4.2]{bri:rat} 
and 
\cite[Corollary 5.6]{go:cells}.  
%
%
\end{proof}

In general, if $X$ is an affine $T$-variety with 
attractive fixed point $x$, and $\ell(x)$ as above 
is finite, then  
$\dim_x X\leq \ell(x)$ \cite[Corollary 1.4.2]{bri:rat}.

\section{$\Q$-filtrable varieties and equivariant Chow groups}

We aim at an inductive description of the equivariant Chow groups 
of filtrable $T$-varieties in the case when the cells are all algebraic rational cells.
Our findings provide purely algebraic analogues of the topological results of \cite{go:cells}.

\begin{dfn}\label{qfiltrable.def}
Let $X$ be a $T$-variety. 
We say that $X$ is {\em $\Q$-filtrable} if 
the following hold: 
\begin{enumerate}
 \item the fixed point set $X^T$ is finite, and
\smallskip
 \item there exists a generic one-parameter subgroup 
$\lambda:\G\to T$ for which the associated $BB$-decomposition
of $X$ is filtrable (Definition \ref{filtable.dfn}) 
and consists of $T$-invariant {\em algebraic rational cells}. 
\end{enumerate}  
\end{dfn}
In particular, $X=\bigsqcup_j X_+(x_j,\lambda)$. 
Also, observe that the fixed points $x_j\in X^T$  
need not be attractive in $X$, 
but they are so in their corresponding 
algebraic rational cells 
$X_+(x_j,\lambda)$. 
The following technical result 
will be of importance in the sequel.

\begin{lem}\label{cell.inj.lem}
If $(X,x)$ is an algebraic rational cell, then  
the equivariant 
multiplicity morphism $e_{X,x}:A^T_*(X)\to \mathcal{Q}$ 
is injective. 
\end{lem}

\begin{proof}
By \cite[Proposition 4.1]{bri:eqchow} 
the map $i_*:A^T_*(x)\to A^T_*(X)$
is injective. Moreover, the image of $i_*$ 
contains $\chi_1\cdots \chi_n A^T_*(X)$, 
where $\chi_i$ are the $T$-weights of $T_xX$. 
Next, recall that $e_x$ is defined as follows: 
given $\alpha\in A^T_*(X)$, we can form the product $\chi_1\cdots \chi_n\alpha$. Thus, 
there exists $\beta \in S$ such that $i_*(\beta)=\chi_1\cdots \chi_n\alpha$. Now let 
$e_x(\alpha)=\frac{\beta}{\chi_1\cdots \chi_n}$. 
Since $A^T_*(X)$ is $S$-free (Proposition \ref{char.ratcell.prop}), 
it is clear from the construction that  
$e_x$ is injective. 
\end{proof}

\smallskip


Let $X$ be a $\Q$-filtrable $T$-variety. 
Then, by assumption, there is a closed algebraic rational cell 
$F=X_+(x_1,\lambda)$ 
(using the order of fixed points induced by the filtration, 
cf. Definition \ref{filtable.dfn}).  
Moreover $U=X\setminus F$ is also $\Q$-filtrable.  
We now proceed to 
describe $A^T_*(X)$ in terms of $A^T_*(F)$ and $A^T_*(U)$. 
Let $j_F:F\to X$ and $j_U:U\to X$ denote the inclusion maps.


\begin{prop}\label{qfiltrable.main}
Notation being as above, 
the maps $j_{F*}:A^T_*(F)\to A^T_*(X)$ and $j^*_{U}:A^T_*(X)\to A^T_*(U)$ fit into the exact sequence  
$$
0\to A^T_*(F)\to A^T_*(X)\to A^T_*(U)\to 0.
$$ 
\end{prop}

\begin{proof} 
It is well-known that the sequence 
$$
\xymatrix{
A^T_*(F)\ar[r]^{j_{F*}}& A^T_*(X)\ar[r]^{j_U^*}& A^T_*(U)\ar[r]& 0 
}
$$
is exact. Thus it suffices to show that $j_{F*}$ is injective. 
But this follows easily from the 
factorization $e_{x,F}=e_{x,X}\circ j_{F*}$. 
Indeed, since $e_{x,F}$ is injective (Lemma  \ref{cell.inj.lem}),  
so is $j_{F*}$.   
\end{proof}

Arguing by induction on the length of the filtration leads to 
the following. 

\begin{thm}\label{qfilt.thm}
Let $X$ be a $\Q$-filtrable $T$-variety. Then 
the $T$-equivariant Chow group of $X$ 
is a free $S$-module of rank $|X^T|$. In fact, 
it is freely generated by the classes of the closures 
of the cells $X_+(x_i,\lambda)$.   
Consequently, $A_*(X)$ is also freely generated by the 
classes of the cell closures 
$\overline{X_+(x_i,\lambda)}$. 

\hfill $\square$ 
\end{thm}

If $X$ is a $\Q$-filtrable variety, 
then each filtered piece $\Sigma_i$ is also $\Q$-filtrable, 
and so Theorem \ref{qfilt.thm} applies at each step 
of the filtration.  
Our approach, 
based on 
equivariant multiplicities, 
is more flexible than 
the general approach which  
compares (equivariant) 
Chow groups with 
(equivariant) homology 
via the (equivariant) cycle map. 
This flexibility will be illustrated 
in the next sections.


\section{Applications to embedding theory} 

We now furnish our theory 
with its first set of examples: 
$\Q$-filtrable embeddings of reductive groups. 
We show that  
the notion of 
algebraic rational cell is well adapted to 
the study of 
group embeddings. 

\smallskip 

\furnotation{
We denote by  
$G$ 
a connected reductive 
linear algebraic group 
with Borel subgroup 
$B$ and maximal torus $T\subset B$.  
We denote by $W$ the Weyl group of $(G,T)$. 


An affine algebraic monoid $M$ is called {\em reductive} it is 
irreducible, normal, and its unit group is a reductive 
algebraic group. See \cite{re:lam} for details. 
Let $M$ be a reductive monoid with zero and unit group $G$.
We denote by $E(M)$ the idempotent set 
of $M$, that is, $E(M)=\{e\in M\,|\, e^2=e\}$. 
Likewise, 
we denote by $E(\overline{T})$ the idempotent set of the associated 
affine torus embedding $\overline{T}$.   
One defines a partial order on $E(\overline{T})$ by declaring 
$f\leq e$ if and only if $fe=f$.
Denote by $\Lambda \subset E(\overline{T})$, the cross section lattice of $M$.
The Renner monoid $\mathcal{R}\subset M$ is a finite monoid
whose group of units is $W$ 
and contains $E(\overline{T})$
as idempotent set. 
Any $x\in \mathcal{R}$ can be written as $x=fu$, 
where $f\in E(\overline{T})$ and $u\in W$. 
Given $e\in E(\overline{T})$, we write $C_W(e)$ for 
the centralizer of $e$ in $W$. 
Denote by $\mathcal{R}_k$ the set of elements of rank $k$ in $\mathcal{R}$, 
that is, 
$\mathcal{R}_k=\{x\in \mathcal{R} \, | \, \dim{Tx}=k \,\}.$
Analogously, one has 
$\Lambda_k\subset \Lambda$ and $E_k\subset E(\overline{T})$. 
For $e\in E(M)$, set  
$M_e:=\overline{\{g\in G\,|\,ge=eg=e\}}$. 
Then $M_e$ is an 
irreducible, normal reductive monoid with 
$e$ as its zero element \cite{bri:m}. 
}

\subsection{Group embeddings} 
A normal irreducible variety $X$ 
is called an {\em embedding} of $G$, 
or a {\em group embedding}, 
if $X$ is a $G\times G$-variety containing an open orbit 
isomorphic to $G$. Due to 
the Bruhat decomposition, group embeddings are spherical 
$G\times G$-varieties. 
Substantial information about the topology of a group 
embedding is obtained by restricting one's attention 
to the induced action of $T\times T$.  
When $G=B=T$, we get back the notion of toric varieties.  
Group embeddings are classified as follows. 

\smallskip

\noindent {\em (I) Affine case:} 
Let $M$ be a reductive monoid with unit group $G$. 
Then $G\times G$-acts naturally 
on $M$ via $(g,h)\cdot x=gxh^{-1}$. 
The orbit of the identity is 
$G\times G/\Delta(G)\simeq G$. 
Thus $M$ is an affine embedding of $G$. 
Remarkably, by a result 
of Rittatore \cite{ri:m},  
reductive monoids are exactly the 
affine embeddings of reductive
groups. 

\smallskip

\noindent{\em (II) Projective case:} 
Let $M$ be a reductive monoid with zero and unit group $G$.
Then there exists a central one-parameter subgroup 
$\epsilon:\mathbb{G}_m^*\to T$,
with image $Z$, 
such that $\displaystyle \lim_{t\to 0}\epsilon(t)=0$. 
Moreover, 
the quotient space $$\P_\epsilon(M):=(M\setminus\{0\})/Z$$
is a normal projective variety on which $G\times G$ acts via
$(g,h)\cdot [x]=[gxh^{-1}].$
Hence, $\P_\epsilon(M)$ is a normal projective embedding of the quotient group $G/Z$.
These varieties were introduced 
by Renner in his study of algebraic monoids 
(\cite{re:hpoly}, \cite{re:ratsm}). 
Notably, normal projective embeddings of  
connected reductive groups are exactly the 
projectivizations of normal algebraic monoids \cite{re:class}. 

\subsection{Algebraic monoids and algebraic rational cells}

\begin{lem}\label{equiv.rel.quot.mon.lem}
Let $\varphi:L\to M$ be a finite surjective 
morphism of normal, 
reductive monoids.  
Then $\varphi$ is the quotient map 
by the finite group 
$\ker(\varphi|_{G_L})$,  
where 
$G_L$ is the unit group of $L$.   
\end{lem}

\begin{proof}
Let $\mu={\rm ker}(\varphi|_{G_L})$. 
Because $\mu$ is a finite and normal subgroup of the connected reductive group $G_L$, 
it is central. 
Hence $\mu\subset T_L$ (for the center of $G_L$ is the intersection of all its maximal tori). 
It follows that the induced map $\tilde{\varphi}:L/\mu\to M$ is bijective and birational. 
But $M$ is normal, so $\tilde{\varphi}$ is an isomorphism.     
\end{proof}

The following is crucial for our purposes.  

\begin{lem}[\protect{\cite[Lemma 1]{gr:r}}]\label{trivialaction.lem}
Let $G$ be a connected linear algebraic group. Let $X$ be a $G$-variety. 
Then the action of $G$ on $A_*(X)$ is trivial. \hfill $\square$
\end{lem}

\begin{cor}\label{mon.rat.chow.cor}
Let $\varphi: L\to M$ be a finite dominant morphism of normal algebraic monoids.
Then $\varphi$ induces an isomorphism of (rational) Chow groups, namely, $A_*(L)\simeq A_*(M)$. 
\end{cor}

\begin{proof}
By Lemma \ref{equiv.rel.quot.mon.lem} 
and \cite{f:int}, Example 1.7.6, 
we have $(A_*(L))^\mu\simeq A_*(M)$. Now, since the action of $\mu$ 
on $A_*(L)$ comes induced from the action of $G_L$ on $A_*(L)$ we 
have $(A_*(L))^{G_L}\subset (A_*(L))^\mu$. But, by 
Lemma \ref{trivialaction.lem}, we have $(A_*(L))^{G_L}=A_*(L)$. 
We conclude that $(A_*(L))^\mu=A_*(L)$. 
\end{proof}

\smallskip

Let $M$ and $N$ be reductive monoids. 
Following Renner \cite{re:hpoly}, we write 
$M\sim_0 N$ if there is a reductive monoid $L$ 
and finite dominant morphisms $L\to M$ and $L\to N$ 
of algebraic monoids. One checks that this gives rise 
to an equivalence relation. 
The following basic result, a consequence of 
Corollary \ref{mon.rat.chow.cor}, states that 
rational Chow groups are an invariant of the 
 equivalence classes.

\begin{cor}\label{mon.are.rcells}
Let $M$ and $N$ be reductive monoids. If $M\sim_0 N$, then 
$A_*(M)\simeq A_*(N)$. \hfill $\square$  
\end{cor}

%

\smallskip

Now let $M$ be a reductive monoid with zero and unit group $G$. 
Recall that $T\times T$ acts on $M$ via $(s,t)\cdot x=txs^{-1}$
and $0$ is the unique attractive fixed point for this action (see e.g. \cite[Lemma 1.1.1]{bri:m}). 
The number of closed irreducible 
$T\times T$-invariant curves in $M$ is finite 
(all of them passing through $0$), and  
it equals $|\mathcal{R}_1|$. Indeed,  
each closed $T\times T$-curve of $M$ 
can be written as $\overline{TxT}$, where $x\in \mathcal{R}_1$, 
for they correspond to the $T\times T$-fixed 
points of $\P_\epsilon(M)$, see \cite[Theorem 3.1]{go:equiv}.  
It follows that $\dim{M}\leq |\mathcal{R}_1|$.   
Similarly, $\overline{T}$ is an affine 
$T$-variety with $0$ as its unique attractive fixed point 
and with finitely many $T$-stable curves. 
The number of these curves equals $|E_1|$ and so
$\dim{T}\leq |E_1|$. 
Next we provide combinatorial criteria 
for showing when $M$ is an algebraic rational cell (for 
the $T\times T$-action). This adds to the list of 
equivalences from \cite{re:hpoly} and \cite{re:ratsm}.

\begin{thm}\label{main.monoid.thm}
Let $M$ be a reductive monoid with zero and unit group $G$. 
Then the following are equivalent.
\smallskip

\begin{enumerate}[(a)]
\item $M\sim_0 \prod_iM_{n_i}(\k)$.

\smallskip

\item If $T$ is a maximal torus of $G$, then 
$\dim{T}=|E_1|$. 

\smallskip

\item $\overline{T}\sim_0 \A^n$. 

\smallskip

\item $(\overline{T},0)$ is an algebraic rational cell.

\smallskip

\item $(M,0)$ is an algebraic rational cell.   

\smallskip

\item $\dim{M}=|\mathcal{R}_1|$. 
\end{enumerate}
\end{thm}

\begin{proof}
The equivalence of $(a)$, $(b)$ and $(c)$ is proven in \cite[Theorem 2.1]{re:hpoly} (no use of rational smoothness is made there). 
The implication $(c) \Rightarrow (d)$ 
follows from Corollary \ref{mon.are.rcells} and 
Proposition \ref{char.ratcell.prop}. 
On the other hand, condition $(d)$ implies $(b)$  
because of Corollary \ref{gkm.cells.cor} and the fact 
that $|E_1(\overline{T})|$ is the number 
of $T$-invariant curves of $\overline{T}$ passing through $0$.  
Hence conditions $(a)$, $(b)$, $(c)$ and $(d)$ are all equivalent. 
\smallskip

Certainly $(a)$ implies $(e)$, 
by Corollary \ref{mon.are.rcells} and  
Proposition \ref{char.ratcell.prop}. In turn, $(e)$ 
yields $(f)$ due to Corollary \ref{gkm.cells.cor} 
and the fact that the number of closed irreducible 
$T\times T$-curves in $M$ equals $|\mathcal{R}_1|$. 
So to conclude the proof it suffices to 
show that $(f)$ implies $(b)$. 
For this we argue as follows. 
\smallskip

Assume $(f)$ and recall that each closed $T\times T$-curve 
in $M$ can be written as $\overline{TxT}$, 
with $x\in \mathcal{R}_1$. 
Moreover, if we write $x=ew$, with $e\in E_1$ and $w\in W$, then $T\times T$ acts on $\overline{TxT}$ through 
the character $(\lambda_e,\lambda_e({\rm int}(w)))$, 
where $\lambda_e:T\to eT\simeq \k^*$ is the character sending 
$t$ to $et$. 

Now, for each $x=ew\in \mathcal{R}_1$, 
we can find a finite $T\times T$-equivariant 
surjective map $\pi_x:\overline{TxT}\to \k_x$.
Here, $T\times T$-acts on $\k_x\simeq \k$ 
via $(\lambda_e,\lambda_e({\rm int}(w)))$. 
Since $\overline{TxT}$ is $T\times T$-invariant and closed in $M$, 
we can extend $\pi_x$ to a $T\times T$-equivariant morphism $\pi_x:M\to \k_x$. 
Syncronizing efforts via the product map, we obtain a $T\times T$-equivariant map 
$$\pi: M\to V=\prod_{x\in \mathcal{R}_1}\k_x, \;\; m\mapsto (\pi_x(m))_{x\in \mathcal{R}_1}.$$  
By construction, $\pi$ is finite (cf. proof of item (\ref{sync.item}) in page \pageref{sync.item}), 
and given that $\dim{M}=|\mathcal{R}_1|$, 
it is also surjective. 

Let $\Delta{T}\subset T\times T$ be the diagonal torus. 
We know that the fixed point set $M^{\Delta{T}}$ equals $\overline{T}$ (see the proof of \cite[Theorem 5.5]{re:lam}). 
Let us look at the restriction map 
$$\pi:\overline{T}\to V^{\Delta{T}}.$$ 
We claim that $\dim{V^{\Delta{T}}}=|E_1(\overline{T})|$.
Indeed, it is clear that for $e\in E_1(\overline{T})$, 
the $T\times T$-invariant curve $\k_e\subset V$ 
is fixed by 
$\Delta{T}$, since $tet^{-1}=e$ (recall that $\overline{T}$ is commutative). 
Hence $$\prod_{e\in E_1}\k_e \subset V^{\Delta{T}}.$$ 
Thus, $$|E_1(\overline{T})|=\dim{\prod_{e\in E_1(\overline{T})}\k_e}\leq \dim{V^{\Delta(T)}}=\dim{T}.$$ 
But, in general, 
$\dim{T}\leq |E_1(\overline{T})|$.
Hence $\dim{T}=|E_1(\overline{T})|$. As this is 
condition $(b)$, the proof is now complete.            
\end{proof}

\begin{rem}\label{converse.curves.cor}
Let $M$ be a reductive monoid with zero.  
Theorem \ref{main.monoid.thm} gives a converse to 
Corollary \ref{gkm.cells.cor}:   
$(M,0)$ is an algebraic rational cell 
if and only if 
$\dim M=|\mathcal{R}_1|$.  
\end{rem}

Theorem \ref{main.monoid.thm} and  
\cite[Theorem 2.4]{re:ratsm} immediately give the following. 
Notice that the cycle map 
is not needed in the proof.  

\begin{cor}\label{rsm.mon.are.alg.cor}
Let $\k=\C$. Let $M$ be a reductive monoid with zero, and 
let $\overline{T}$ be the associated affine toric variety.  
Then $M$ (resp. $\overline{T}$) 
is rationally smooth 
if and only if 
$M$ (resp. $\overline{T}$) is an algebraic rational cell.   \hfill $\square$ 
\end{cor}

\subsection{$\Q$-filtrable projective group embeddings}
We start by recalling \cite[Definition 2.2]{re:hpoly}.  

\begin{dfn}\label{chow.rs.emb.dfn}
A reductive monoid $M$ with zero element 
is called {\em quasismooth} 
if, 
for any 
minimal non-zero idempotent $e\in E(M)$,  
$M_e$ satisfies the 
conditions of Theorem \ref{main.monoid.thm}. 
\end{dfn}

In other words, $M$ 
is quasismooth if and only if 
$M_e$ is an algebraic rational cell, for 
any minimal non-zero idempotent $e\in E(M)$. 

\smallskip

Now consider the projective 
group embedding $\P_\epsilon(M)=(M\setminus \{0\})/Z$ (as in Section 5.1).  
When $\k=\C$,  it is worth noting that  
$M$ is quasismooth if and only if 
$\P_\epsilon(M)$ is rationally smooth  
\cite[Theorem 2.5]{re:ratsm}.

\smallskip

Next is the second main result of this section. 
It is an extension of 
\cite[Theorem 7.4]{go:cells}
to equivariant Chow groups.

\begin{thm}\label{ratsm.emb.are.qfilt.thm} 
Let $M$ be a reductive monoid with zero. 
If $M$ is quasismooth, then the 
projective group embedding $\P_\epsilon(M)$ 
is $\Q$-filtrable (as in Section 4). 
\end{thm}

\begin{proof}
The strategy is to adapt the proof of 
\cite[Theorem 7.4]{go:cells} in light of 
Proposition \ref{char.ratcell.prop} and 
Theorem \ref{main.monoid.thm}. 
Recall that, by \cite[Theorem 3.4]{re:hpoly}, 
$\P_\epsilon(M)$ comes equipped with a 
BB-decomposition 
$$
\P_\epsilon(M)=\bigsqcup_{r\in \mathcal{R}_1} C_r,
$$
where $\mathcal{R}_1$ identifies to $\P_\epsilon(M)^{T\times T}$. 
(In fact these cells are $B\times B$-invariant, where 
$B$ is a Borel subgroup of $G$.)  
Given that $\P_\epsilon (M)$ is normal,  
projective, and $\mathcal{R}_1$ is finite, 
this BB-decomposition is filtrable 
(Theorem \ref{bbdecomp.thm}). 
So we just need to show that these cells  
are algebraic rational cells. 
Furthermore, since the $C_r$ 
are affine $T\times T$-varieties 
with an attractive fixed point $[r]$,  
Proposition \ref{char.ratcell.prop} reduces the 
proof to showing that $A_*(C_r)\simeq \Q$.   

\smallskip

Bearing this in mind, we delve a bit further 
into the structure of these cells. 
By \cite[Lemma 4.6 and Theorem 4.7]{re:hpoly}, 
each $C_r$ is isomorphic to $$U_1\times C_r^*\times U_2,$$ 
where the $U_i$ are affine spaces. Moreover, writing 
$r\in \mathcal{R}_1$ as $r=ew$, with $e\in E_1(\overline{T})$
and $w\in W$, yields $C_r^*=C_e^*w$. 
Hence, 
%
by the Kunneth formula (which holds, because the $U_i$'s 
are affine spaces), we are further reduced 
to showing that $A_*(C_e^*)\simeq \Q$, 
for $e\in E_1(\overline{T})$. 

\smallskip

Now we call the reader's attention to \cite[Theorem 5.1]{re:hpoly}. 
It states that if $M$ is quasismooth, then 
$$
C_e^*=f_eM(e)/Z,
$$
for some unique $f_e\in E(\overline{T})$, where $M(e)=M_eZ$, 
and $M_e$ is reductive monoid with $e$ as its zero. 
By hypothesis, we know that 
$M_e$ is an algebraic rational cell, that is, 
$A_*(M_e)\simeq \Q$. 
Since $M(e)/Z$ is a reductive monoid with $[e]$ as 
its zero, and $M_e\sim_0 M(e)/Z$, 
Corollary \ref{mon.are.rcells} 
yields $A_*(M(e)/Z)\simeq \Q$.  
Now, by \cite[Lemma 1.1.1]{bri:m}, one can 
find a one-parameter subgroup $\lambda:\G\to T$, 
with image $S$, such that $\lambda(0)=f$ and   
$$f_eM(e)/Z=(M(e)/Z)^S.$$ 
That is, 
$f_eM(e)/Z$ is the $S$-fixed point set of $M(e)/Z$. 
But now we invoke Lemma \ref{dim.fix.set.lem} to get 
$$\dim_\Q A_*((M(e)/Z)^S)\leq \dim_\Q A_*(M(e)/Z)=1.$$
Hence, {\em a fortiori} 
$$\dim_\Q A_*((M(e)/Z)^S)=\dim_\Q A_*(f_eM(e)/Z)=
\dim_\Q A_*(C_e^*)=1.$$   
This shows that $A_*(C^*_e)=\Q$, concluding the argument. 
\end{proof}

It is well-known that for projective simplicial  
toric varieties (equivalently, rationally smooth 
projective toric varieties) the equivariant cycle map 
is an isomorphism over $\Q$. 
Below we extend this result to 
all rationally smooth projective group embeddings.

\begin{cor}\label{cycle.mon.cor}
Let $\k=\C$. If $M$ is a quasismooth monoid with zero, 
then the equivariant cycle map   
$$cl^T_{\P_\epsilon(M)}: A_*^T(\P_\epsilon(M))\to H^T_*(\P_\epsilon(M))$$  
is an isomorphism of free $S$-modules. 
%
Moreover, the usual cycle map 
$$
cl_{\P_\epsilon(M)}: A_*(\P_\epsilon(M))\to H_*(\P_\epsilon(M))
$$
is an isomorphism of $\Q$-vector spaces.   
\end{cor}

\begin{proof}
By 
\cite[Theorem 7.4]{go:cells} 
$\P_\epsilon(M)$ has no homology in odd degrees, 
and each cell is rationally smooth, so 
$H_{*,c}(C_r)\simeq \Q$ and $H^T_{*,c}(C_r)\simeq S$. 
Now Theorem \ref{ratsm.emb.are.qfilt.thm}
implies that the cycle maps 
$cl_{C_r}:A_*(C_r)\to H_{*,c}(C_r)$ and 
$cl^T_{C_r}:A^T_*(C_r)\to H^T_{*,c}(C_r)$ 
are isomorphisms. 
Arguing by induction on the length of the 
filtration concludes the proof. 
\end{proof}

In \cite{re:hpoly} and \cite{re:ratsm}, Renner has 
computed the $H$-polynomial of a quasismooth monoid. 
This polynomial counts the number of 
algebraic rational cells (of each dimension) 
that appear in the BB-decomposition of $\P_\epsilon(M)$.
In particular, when $\P_\epsilon(M)$ is simple 
(i.e. it contains a unique $G\times G$-orbit), 
\cite{re:hpoly} shows that the number of cells 
of dimension $k$ equals the number of cells of dimension $n-k$, 
where $n=\dim \P_\epsilon(M)$. 
By Theorem \ref{qfilt.thm}, this yields  
$$\dim_\Q A_k(\P_\epsilon(M))=\dim_\Q A_{n-k}(\P_\epsilon(M)).$$ 
Over the complex numbers, this is equivalent to the fact that 
$\P_\epsilon(M)$ satisfies Poincar\'e duality for 
rational singular cohomology. 

\section{Connections with topology: spherical varieties} 
In this section we work over the complex numbers.  
The aim is to relate the results of this paper 
with those of \cite{go:cells} in 
the case of spherical varieties. 

\smallskip

The following result is a particular case 
of \cite[Theorem 3]{to:linear}.  

\begin{thm}\label{totaro.cycle.thm}
Let $\Gamma$ be connected solvable linear algebraic group. 
For any $\Gamma$-variety $Y$ with a finite number of orbits, 
the natural map 
$$A_i(Y)\otimes \Q\longrightarrow W_{-2i}H^{BM}_{2i}(Y,\Q),$$ from the Chow groups into the smallest subspace of 
Borel-Moore homology with respect to the weight filtration 
is an isomorphism. \hfill $\square$
\end{thm}

\smallskip

Let $G$ be a connected reductive 
linear algebraic group with Borel 
subgroup $B$ and maximal torus $T\subset B$.  
Recall that given a one-parameter subgroup 
$\lambda:\C^*\to T$, we can define  
$$G(\lambda)=\{g\in G\,|\,\lambda(t)g\lambda(t)^{-1} \,
{\rm \;has\; a\; limit\; as \;}t\to 0\}.$$
It is well-known that $G(\lambda)$ is a parabolic 
subgroup of $G$ with unipotent radical 
$
R_uG(\lambda)=\{g\in G\,|\,\lim_{t\to 0}\lambda(t)g\lambda(t)^{-1}=1\}.
$
Moreover, the centralizer $C_G(\lambda)$ of the image 
of $\lambda$ is connected,  
and the product 
morphism $R_uG(\lambda)\times C_G(\lambda)\to G(\lambda)$
is an isomorphism of varieties. Also, 
the parabolic subgroups $G(\lambda)$
and $G(-\lambda)$ are opposite. 
Finally, $G(\lambda)=B$ 
if and only if $\lambda$ lies 
in the interior of the Weyl chamber 
associated with $B$. 
See e.g. \cite[Theorem 13.4.2]{sp:lag}.

\smallskip

Next we show that algebraic rational 
cells are naturally found 
on rationally smooth spherical varieties. 

\begin{thm}\label{sphrsm_implies_arc.lem}
Let $X$ be a $G$-spherical variety 
with an attractive $T$-fixed point $x$. 
Let $X_x$ be the unique open affine 
$T$-stable neighborhood of $x$. 
If $X$ is rationally smooth at $x$, then 
$(X_x,x)$ is an algebraic rational cell. 
\end{thm}

\begin{proof}
Because $x$ is attractive, we may choose 
$\lambda$ such that 
$X_x=X_+(x,\lambda)$ and $G(\lambda)=B$. 
Since $X$ is 
rationally smooth at $x$, so 
is the open subset $X_x$. Moreover, 
$X_x$ is rationally 
smooth everywhere, and so 
$X_x$ is a rational cell \cite[Definition 3.4]{go:cells}. 
By Theorem \ref{totaro.cycle.thm} 
we have 
$$A_i(X_x)\simeq W_{-2i}H^{BM}_{2i}(X_x,\Q)
\simeq H^{2i}_c(X_x,\Q)=
\left\{
\begin{array}{cc}
\Q & {\rm \;if \,} i=\dim_{\C}X_x\\
 0 & {\rm \;otherwise,}
\end{array}
\right.$$
where the last two identifications   
follow from the fact that 
$X_x$ is a cone over a rational cohomology 
sphere \cite[Corollary 3.11]{go:cells}.  
\end{proof}

\smallskip

Let $X$ be a $G$-spherical variety. Recall 
that $X^T$ is finite. For convenience, 
we use the following nomenclature. We say that 

\begin{enumerate}[(a)]
\item $X$ has an 
{\em algebraic} $\Q$-filtration, 
if it satisfies Definition \ref{qfiltrable.def} 
for some generic one-parameter subgroup $\lambda$ of $T$. 

\item 
$X$ has a 
{\em topological} $\Q$-filtration, 
if there exists a generic 
one-parameter subgroup 
$\lambda:\C^*\to T$ for which the associated 
BB-decomposition of $X$ is filtrable, and consists of 
{\em rational cells} \cite{go:cells}.
\end{enumerate}

%

\begin{thm}\label{sp_rs_impl_crs}
Let $X$ be a $G$-spherical variety. 
If $X$ has a topological $\Q$-filtration, 
then this filtration is also an algebraic $\Q$-filtration. 
\end{thm}

\begin{proof}
Let $X^T=\{x_1,\ldots, x_m\}$. By assumption, there exists 
a generic one-parameter subgroup such that $X^{\lambda}=X^T$,  
and the cells $X_j:=X_+(x_j,\lambda)$ 
are rational cells. 
Consider the parabolic subgroup $G(\lambda)$. 
We claim that 
the cells $X_j$ 
are invariant under $G(\lambda)$. 
Indeed, 
$G(\lambda)=R_u(\lambda)\times C_G(\lambda)$, 
and $C_G(\lambda)$, being connected, fixes 
each $x_j\in X^\lambda$. 
Now let $x\in X_j$
and write $g\in G(\lambda)$ as $g=uh$, 
with $u\in R_u(\lambda)$ and $h\in C_G(\lambda)$.  
Then 
$$\lambda(t)g\cdot x=
\lambda(t)uh\lambda(t)^{-1}\lambda(t)\cdot x=
\lambda(t)u\lambda(t)^{-1}h\lambda(t)\cdot x.
$$  
Taking limits at $0$ gives the claim. 
Because $X$ is spherical, it contains only 
finitely many orbits of any Borel subgroup 
of $G$. Therefore, a Borel subgroup of $G(\lambda)$
has finitely many orbits in $X_j$. 
%
%
Applying Theorem \ref{totaro.cycle.thm} 
to each $X_j$ yields 
$$A_i(X_j)\simeq 
W_{-2i}H^{BM}_{2i}(X_j,\Q)
\simeq H^{2i}_c(X_j,\Q)=
\left\{
\begin{array}{cc}
\Q \;& {\rm if \,} i=\dim_{\C}{X_j}\\
 0 \;& {\rm otherwise,}
\end{array}
\right.
$$
noting that 
$X_j$ is a cone over a rational cohomology 
sphere \cite[Corollary 3.11]{go:cells}. 
Therefore, by Lemma \ref{contractible.lem}, 
the cells $X_j$ are algebraic rational cells. 
%
This concludes the proof. 
\end{proof}

\begin{rem}
Let $X$ be a $G$-spherical variety, and let $\lambda$ 
be a generic one-parameter subgroup. Then the argument 
above shows that the cells $X_+(x_j,\lambda)$ are 
$T'$-linear varieties, where $T'\subset G(\lambda)$ 
is a maximal torus of $G$.   
\end{rem}

%
%
%

Arguing by induction on the length of the filtration, 
using the fact that a $T$-variety with a 
topological $\Q$-filtration has no (co)homology 
in odd degrees, gives immediately the following. 

\begin{cor}\label{qfiltrable.sph.thm}
Let $X$ be a spherical $G$-variety with a  
topological $\Q$-filtration, say  
$
\emptyset=\Sigma_0\subset \Sigma_1\subset \ldots \subset \Sigma_m=X.  
$ 
Then, for every $j$, both cycle maps, 
$cl_{\Sigma_j}:A_*(\Sigma_j)\to H_{*}(\Sigma_j)$ 
and 
$cl^T_{\Sigma_j}:A^T_*(\Sigma_j)\to H^T_{*}(\Sigma_j)$
are isomorphisms. \qed
\end{cor}

We should remark that Theorem \ref{sp_rs_impl_crs} 
provides another proof of Theorem \ref{ratsm.emb.are.qfilt.thm}. 
However, in the case of group embeddings, 
the approach taken in Section 5 is more intrinsic,  
for it uses the rich structure of the Chow groups and 
the fine combinatorial structure of algebraic monoids. 
Notice that the results of Section 5 are independent of 
Theorem \ref{totaro.cycle.thm}. This shows how the 
notion of algebraic rational cells 
is well adapted to embedding theory, and 
opens the way for further work in this direction. 
For instance, the results of this paper, together with those of \cite{go:tlin}, 
yield some characterizations of Poincar\'e duality for the 
equivariant operational Chow {\em rings} 
of projective group embeddings, and related spherical varieties \cite{go:poinc}. This will appear elsewhere. 

\smallskip

Finally, observe 
that, when looking for concrete 
examples, topological $\Q$-filtrations 
are slightly more approachable, for 
they are built using the classical topology 
of a complex variety, 
and could 
be obtained e.g. via 
Hamiltonian actions.  
Our Theorem \ref{sp_rs_impl_crs} 
guarantees that 
the topological knowledge 
thus acquired  
gets transformed into algebraic 
information about the Chow groups. 
This provides examples of singular 
spherical varieties for which the 
cycle map is an isomorphism    
(e.g. rationally smooth 
group embeddings). 
It is worth noting, however,   
that the study of 
algebraically $\Q$-filtrable 
varieties can be carried out 
intrinsically, via equivariant intersection theory. 

\end{document}